\documentclass[12pt, twoside, leqno]{article}
\usepackage{cite}
\usepackage{indentfirst}
\usepackage[T1]{fontenc}
\usepackage{enumitem}
\usepackage{amsthm,amsmath}
\usepackage{amssymb}
\newtheorem{thm}{Theorem}[section]
\newtheorem{lemma}{Lemma}[section]
\newtheorem{cor}{Corollary}[section]
\newtheorem*{xrem}{Remark}
\numberwithin{equation}{section}
\begin{document}
	\allowdisplaybreaks[3]
	\baselineskip=17pt
  \title{\bf Some congruences involving  generalized Bernoulli numbers \\
  	and Bernoulli polynomials}
  
   \author{Ni Li\\
   	 School of Mathematics and Statistics,\\
   	 Northwestern
   	 Polytechnical University\\
   	Xi'an, Shaanxi, 710072,\\
   	People's Republic of China\\
   	E-mail: 2741227558@qq.com\\
   	\and 
   	Rong Ma\\
   	School of Mathematics and Statistics, \\
   	Northwestern
   	Polytechnical University\\
   	Xi'an, Shaanxi, 710072,\\
   	People's Republic of China\\
   	E-mail: marong@nwpu.edu.cn}
   	 \date{}
   	\maketitle
   	\renewcommand{\thefootnote}{}
   	\footnote{This work is supported by Natural Science Basic Research Project of Shaanxi Province (2021JM-044).} 
   	\renewcommand{\thefootnote}{\arabic{footnote}}
   	\setcounter{footnote}{0}
   	
   	\begin{abstract}
   		Let $[x]$ be the integral part of $x$, $n>1$ be a positive integer and $\chi_n$ denote the trivial Dirichlet character modulo $n$. In this paper, we use an identity established by Z. H. Sun to  get congruences of $T_{m,k}(n)=\sum_{x=1}^{[n/m]}\frac{\chi_n(x)}{x^k}\left(\bmod n^{r+1}\right)$ for $r\in \{1,2\}$, any positive integer $m $ with $n \equiv \pm 1 \left(\bmod m \right)$ in terms of  Bernoulli polynomials. As its an application, we also obtain some new congruences  involving binomial coefficients modulo $n^4$ in terms of generalized Bernoulli numbers.
   	\end{abstract}
	\section{Introduction}
	Let $n$, $r\ge 2$ be positive integers with $(n,r)=1$  and $q_r(n)$ denote the Euler quotient, i.e., 
    $$
     q_r(n)=\frac{r^{\phi(n)}-1}{n},$$
	where $\phi$ is the Euler totient function. For the trivial Dirichlet character $\chi_n$ modulo $n$, the generalized Bernoulli numbers $\{B_{s,\chi_n}\}$ are defined by 
	$$
	\sum_{x=1}^{n}\frac{\chi_n(x)te^{xt}}{e^{nt}-1}
	=\sum_{s=0}^{\infty}B_{s,\chi_n}\frac{t^s}{s!}.
	$$
	 Moreover, we have  the identity $B_{s,\chi_n}=B_s\prod_{p \mid n}(1-p^{s-1})$, where $B_s$ is the $s$th Bernoulli number. On the other hand, Bernoulli polynomials $\{B_s(x)\}$ are defined by $$
	B_s(x)=\sum_{k=0}^{s}\left(\begin{array}{c}
		s \\
     	k
	\end{array}\right)B_kx^{s-k}\quad (s\ge 0).
	$$\par 
	It is interesting  to investigate congruences of sums $T_{m,k}(n)=\sum_{x=1}^{[n/m]}\frac{\chi_n(x)}{x^k}$ involving the Euler quotient, generalized Bernoulli numbers 
	and Bernoulli polynomials  modulo integer powers. For example, in 1938, Lehmer (see \cite{E.L}) established the famous  congruence
	 \begin{equation}\tag{1.1}
	\sum_{i=1}^{(p-1) / 2} \frac{1}{i} \equiv-2 q_2(p)+p q_2^2(p) \quad\left(\bmod p^2\right)
	\end{equation}
   for any odd prime $p$. Applying this congruence and other similar congruences, Lehmer obtained some results in studying the Fermat's Last Theorem (see \cite{P.R}). The proof of (1.1) mainly relied on Bernoulli polynomials of fractional arguments. In 2002, Cai (see \cite{Cai}) used an identity involving generalized Bernoulli numbers  proved by Szmidt, Urbanowicz and  Zagier (see \cite[(6)]{SUZ}) and obtained a more general congruence for any odd integer $n > 1$, that is
		\begin{equation}\tag{1.2}
		\sum_{i=1}^{(n-1) / 2} \frac{\chi_n(i)}{i} \equiv-2 q_2(n)+nq_2^2(n) \quad\left(\bmod n^2\right).
	\end{equation}\par
   In 2008, Sun (see \cite{Szh}) determined  the following congruence 
   \begin{equation}\tag{1.3}
   	\sum_{x=1}^{[p / 4]} \frac{1}{x^2} \equiv(-1)^{\frac{p-1}{2}}\left(8 E_{p-3}-4 E_{2 p-4}\right)+\frac{14}{3} p B_{p-3} \quad\left(\bmod p^2\right)
   \end{equation}
  for prime $p > 5$, where $E_n$ is the $n$th Euler number (see \cite{Szh07}). In 2012, Kanemitsu, Urbanowicz and Wang  (see \cite{KUW}) generalized Sun's congruence (1.3) by using Cai's technique and showed that
\begin{multline}\tag{1.4}
		\sum_{0<i<n / 4} \frac{\chi_n(i)}{i^2} \equiv 8\left(n B_{n \phi(n)-2} \prod_{p \mid 4 n}\left(1-p^{n \phi(n)-3}\right)\right.\\
	\left. +\frac{1}{2}(-1)^{(n-1) / 2} E_{n \phi(n)-2} \prod_{p \mid n}\left(1-(-1)^{(p-1) / 2} p^{n \phi(n)-2}\right)\right)\quad (\bmod n^2)
	\end{multline}
     for positive odd integer $n>3$. Later in 2015, Kanemitsu, Kuzumaki and Urbanowicz (see \cite{KKU}) also adopted the identity proved in \cite{SUZ} and obtained some new congruences of the sums $T_{m,k}(n)=\sum_{x=1}^{[n/m]}\frac{\chi_n(x)}{x^k}\left(\bmod n^{r+1}\right)$ for $r\in \{0,1,2\}$, all divisors $m$ of 24. Recently in 2019, Cai, Zhong and Chern (see \cite{CZC}) showed $T_{m,2}(n) \left(\bmod n\right)$ and $T_{m,1}(n) \left(\bmod n^2\right)$ for $m \in \{2,3,4,6\}$ by introducing a new generalization of Euler's totient. Congruences for sums $T_{m,k}(n)=\sum_{x=1}^{[n/m]}\frac{\chi_n(x)}{x^k}$  not only play an important role in studying the first case of Fermat's last theorem, but also help to generalize some congruences involving binomial coefficients.\par
     For example, in 2002, using (1.2), Cai (see \cite[Theorem 2]{Cai}) established the following congruence for an arbitrary positive integer $n>1$,
     \begin{equation}\tag{1.5}
     	\prod_{d \mid n}\left(\begin{array}{c}
     		d-1 \\
     		(d-1) / 2
     	\end{array}\right)^{\mu(n / d)} \equiv(-1)^{\phi(n) / 2} 4^{\phi(n)} \begin{cases}\left(\bmod n^3\right) & \text { if } 3 \nmid n, \\
     		\left(\bmod n^3 / 3\right) & \text { if } 3 \mid n,\end{cases} 
     \end{equation}
    which is a generalization of the well-known Morley's congruence (see \cite{M})
    \begin{equation}\tag{1.6}
    	(-1)^{(p-1) / 2}\left(\begin{array}{c}
    		p-1 \\
    		(p-1) / 2
    	\end{array}\right) \equiv 4^{p-1}\quad(\bmod p^3)
    \end{equation}
   for any prime $p\ge 5$.
   Morley's congruence has a profound impact in combinatorial number theory. In 2008, applying (1.3), Sun (see \cite[Theorem 3.8]{Szh}) proved
    $$
   (-1)^{\left[\frac{p}{4}\right]}\left(\begin{array}{c}
   	p-1 \\
   	{\left[\frac{p}{4}\right]}
   \end{array}\right) \equiv 1+3 p q_2(p)+p^2\left(3 q_2(p)^2-(-1)^{\frac{p-1}{2}} E_{p-3}\right) \quad\left(\bmod p^3\right)
   $$
   for prime $p>5$. In 2019, Cai, Zhong and Chern (see  \cite[Theorem 1.2]{CZC}) showed that 
    \begin{equation}\tag{1.7} 
    	\prod_{d \mid n}\left(\begin{array}{c}
    		kd-1 \\
    		(d-1) / 2
    	\end{array}\right)^{\mu(n / d)} \equiv(-1)^{\phi(n) / 2} 4^{k\phi(n)} \begin{cases}\left(\bmod n^3\right) & \text { if $3 \nmid n$,}  \\
    		\left(\bmod n^3 / 3\right) & \text { if $3 \mid n,$} \end{cases}
    \end{equation}
      for any positive integer $k$ and odd integer $n>1$.  \par
     In this paper, we study $T_{m,k}(n)=\sum_{x=1}^{[n/m]}\frac{\chi_n(x)}{x^k} $ by applying the identity proved by Sun (see \cite[Theorem 2.1]{Szh07}), which is 
     \begin{equation}\tag{1.8}
     \sum_{\substack{x=0 \\ x \equiv r(\bmod m)}}^{p-1} x^k=\frac{m^k}{k+1}\left(B_{k+1}\left(\frac{p}{m}+\left\{\frac{r-p}{m}\right\}\right)-B_{k+1}\left(\left\{\frac{r}{m}\right\}\right)\right),
     \end{equation}
     where $p$, $m \in \mathbb{N}$, $k$, $ r \in \mathbb{Z}$ with $k \geqslant 0$ and $\{x\}$ is the decimal part of $x$. We get   congruences of $T_{m,k}(n)=\sum_{x=1}^{[n/m]}\frac{\chi_n(x)}{x^k} \left(\bmod n^{r+1}\right)$ for $r\in \{1,2\}$, any positive integer $m \ge 2$ with  $n \equiv \pm 1 \left(\bmod m\right)$. Moreover, we use  these congruences to get some new congruences  involving binomial coefficients modulo $n^4$. For example, for any positive integers $k\ge 1$,  $n>1$ with $(n,6)=1$, we have
     \begin{multline}
     \prod_{d \mid n}\left(\begin{array}{c}
     	kd-1 \\
     	(d-1) / 2
     \end{array}\right)^{\mu(n / d)}\\ 
     \quad  \equiv(-1)^{\frac{\phi(n)}{2}}\left\{4^{k\phi(n)}-\frac{k(7-14k+8k^2)}{4}n^3\frac{B_{\phi(n)-2,\chi_n}}{\phi(n)-2}\right\}\quad (\bmod n^4)\notag
     \end{multline}
     and
     \begin{multline}
     \prod_{d \mid n}\left(\begin{array}{c}
     	kd+[vd /2] \\
     	(vd-1) / 2
     \end{array}\right)^{\mu(n / d)}\\
      \equiv4^{-k \phi(n)}+\frac{k(7v^2+14kv+8k^2)}{4}n^3\frac{B_{\phi(n)-2,\chi_n}}{\phi(n)-2}\quad (\bmod n^4),\notag
     \end{multline}
      where $v$ is a positive odd integer. At last, using $T_{1,k}(n)\left(\bmod n^{r+1}\right)$ for $r\in \{1,2\}$, we have
      \begin{equation}
      	\prod_{d \mid n}\left(\begin{array}{l}
      		u d \\
      		v d
      	\end{array}\right)^{\mu(n / d)} \equiv 1+\frac{uv(v-u)}{2}B_{\phi (n^3)-2,\chi_n}n^3 \quad (\bmod n^4),\notag
      \end{equation}	
       where $n >1$ is a positive integer with $(n,6)=1,$ $u$, $v$ are positive integers  with $u>v$.
      \section{Basic lemmas}
       In this section, we introduce the following lemmas will be used later. We begin with the divisibility of Bernoulli polynomials.
     \begin{lemma}[See {\cite{Szh}}]
     Suppose that $k$, $p \in \mathbb{N}$ with $p>1$. If $x$, $y \in \mathbb{Z}_p$, then $p B_k(x) \in \mathbb{Z}_p$ and $\left(B_k(x)-\right.$ $\left.B_k(y)\right) / k \in \mathbb{Z}_p$. If $p$ is an odd prime such that $p-1 \nmid k$, then $B_k(x) / k \in \mathbb{Z}_p$.
     \end{lemma}
     \begin{xrem}
     \rm{In particular, for $x=0$, we get the von Staudt-Clausen theorem (see \cite{A2002}).} 
     \end{xrem}
     \begin{lemma}[See {\cite{MOS}}]
     	 Let $x$ and $y$ be variables and $n \in \mathbb{N}^+$. Then
     	 \begin{enumerate}[label=\upshape(\roman*), leftmargin=*, widest=iv]
     	 	\item $B_{2 n+1}=0\left(n \ge 1\right)$,\label{it:1}
     	 	\item $B_n(1-x)=(-1)^n B_n(x),$\label{it:2}
     	 	\item $B_n(x+y)=\sum_{r=0}^n\left(\begin{array}{l}n \\ r\end{array}\right) B_{n-r}(y) x^r$,\label{it:3}
     	 	\item $E_{n-1}(x)=\dfrac{2^n}{n}\left(B_n\left(\dfrac{x+1}{2}\right)-B_n\left(\dfrac{x}{2}\right)\right)$.\label{it:4}	
     	 \end{enumerate}
     \end{lemma}
     \begin{lemma}[See {\cite{MOS}}]
     Let $n \in \mathbb{N}$, then
     $$
     B_{2 n}\left(\frac{1}{4}\right)=B_{2 n}\left(\frac{3}{4}\right)=\frac{2-2^{2 n}}{4^{2 n}} B_{2 n},\quad 
     B_{2 n}\left(\frac{1}{3}\right)=B_{2 n}\left(\frac{2}{3}\right)=\frac{3-3^{2 n}}{2 \cdot 3^{2 n}} B_{2 n}
     $$
     and$$
     \quad B_{2 n}\left(\frac{1}{6}\right)=B_{2 n}\left(\frac{5}{6}\right)=\frac{\left(2-2^{2 n}\right)\left(3-3^{2 n}\right)}{2 \cdot 6^{2 n}} B_{2 n} .
     $$
     \end{lemma}
     \begin{lemma}
     Let $p>3$ be a prime, $d$, $k$, $m$, $l\in \mathbb{Z}^+$ with $1 \le k<{\phi(p^{3l})-3}$ and $p\nmid m$, $p\nmid d$, $s'=\phi(p^{3l})-k$, then
     \begin{multline}
     	\sum_{\substack{x=1 \\p\nmid x}}^{[dp^l/m]} \frac{1}{x^k}
     	 \equiv 
     		\begin{aligned}
     		\begin{cases}
     		\frac{B_{s'+1}\left(\left\{\frac{-dp^l}{m}\right\}\right)}{s'+1}+B_{s'}\left(\left\{\frac{-dp^l}{m}\right\}\right)\frac{dp^l}{m}\\
     		\quad -\frac{k}{2}B_{s'-1}\left(\left\{\frac{-dp^l}{m}\right\}\right)\frac{(dp^l)^2}{m^2}\quad(\bmod p^{3l})& \text {if $2\mid k$, $p^l\in {I(k,2)}$,} \\
     		\frac{B_{s'+1}\left(\left\{\frac{-dp^l}{m}\right\}\right)-B_{s'+1}}{\phi(s'+1}+B_{s'}\left(\left\{\frac{-dp^l}{m}\right\}\right)\frac{dp^l}{m}\\
     		\quad -\frac{k}{2}B_{s'-1}\left(\left\{\frac{-dp^l}{m}\right\}\right)\frac{(dp^l)^2}{m^2}\quad (\bmod p^{3l})& \text {if $2\nmid k$,} 
     	\end{cases}	
     		\end{aligned} \notag
   \end{multline}
     where $I(k,\gamma)=\{n > 1: p-1 \nmid k+\gamma, \text{if $p \mid n$}\}.$
     \end{lemma}
       \begin{proof}
        By Euler's theorem and $k<{\phi(p^{3l})-3}$, we have
       \begin{equation}\label{2.1}
       \sum_{\substack{x=1 \\p\nmid x}}^{[dp^l/m]} \frac{1}{x^k}\equiv \sum_{\substack{x=1 \\p\nmid x}}^{[dp^l/m]} x^{\phi(p^{3l})-k}\equiv \sum_{x=1}^{[dp^l/m]}x^{\phi(p^{3l})-k}\quad \left(\bmod p^{3l}\right).
       \end{equation}
   
       Taking $p=dp^l$ and $r=0$ in (1.8), from Lemma 2.2(iii), we have
       \begin{multline}
       	\sum_{x=1}^{[dp^l/m]}x^{\phi(p^{3l})-k}\\
       	\begin{aligned}
       		 &{}=\frac{1}{\phi(p^{3l})-k+1}\left(B_{\phi(p^{3l})-k+1}\left(\frac{dp^l}{m}+\left\{\frac{-dp^l}{m}\right\}\right)-B_{\phi(p^{3l})-k+1}\right)\\
       		&{}=\frac{B_{\phi(p^{3l})-k+1}\left(\left\{\frac{dp^l}{m}\right\}\right)-B_{\phi(p^{3l})-k+1}}{\phi(p^{3l})-k+1}  +B_{\phi(p^{3l})-k}\left(\left\{\frac{-dp^l}{m}\right\}\right)\frac{dp^l}{m}\\
       		&\quad  -\frac{k}{2}B_{\phi(p^{3l})-k-1}\left(\left\{\frac{-dp^l}{m}\right\}\right)\frac{(dp^l)^2}{m^2}\\
       		& \quad +\sum_{i=3}^{\phi(p^{3l})-k+1}\left(\begin{array}{l}\phi(p^{3l})-k \\ \quad\quad i-1\end{array}\right)\frac{p^{(i-3)l}}{i}B_{\phi(p^{3l})-k+1-i}\left(\left\{\frac{-dp^l}{m}\right\}\right)\frac{d^ip^{3l}}{m^i}.
       	\end{aligned}\notag
       \end{multline}
       When $2\mid k$, $p^l\in {I(k,2)}$ or $2 \nmid k$, we have $p-1\nmid k+2$. From Lemma 2.1 we see that $B_{\phi(p^{3l})-k-2}\left(\left\{\frac{-dp^l}{m}\right\}\right) \in \mathbb{Z}_p$,  $ \frac{p^{(i-3)l}}{i}B_{\phi(p^{3l})-k+1-i}\left(\left\{\frac{-dp^l}{m}\right\}\right)\in \mathbb{Z}_p$ for
       $i \ge 4$. Let $s'=\phi(p^{3l})-k$. Now combining with (2.1) and noticing $B_{2n+1}=0(n\ge 1)$, we obtain Lemma 2.4.
       \end{proof}
	\begin{lemma}
	Let $p \ge 3$ be a prime, $x\in\mathbb{Z}_p$, $\alpha$, $k$, $l\in \mathbb{Z}^+$ such that $k$, $l$ are even, $\alpha \le k-1 \le l-1$, and $k$, $l$ are not divisible by $p-1$, if  $k\equiv l\left (\bmod \phi(p^\alpha )\right),$ then
	$$
	\frac{B_k(x)}{k}\equiv \frac{B_l(x)}{l}\quad(\bmod p^\alpha ).$$
	As $x=0$, it becomes Kummer's congruence (see \cite{IR}).
	\end{lemma}

    \begin{proof}
   we recall an identity given by Sun, Z. W. (see \cite[Corollary 1.3]{Szw}). Let $a \in \mathbb{Z}$, $k$, $q$, $m\in\mathbb{Z}^+$, $(m,q)=1$, then
   \begin{multline}
    \frac{1}{k}\left(m^kB_k\left(\frac{x+a}{m}\right)-B_k(x)\right)\\
    \equiv \sum_{j=0}^{q-1}\left(\left[\frac{a+jm}{q}\right]+\frac{1-m}{2}\right)\left(x+a+jm\right)^{k-1}\quad (\bmod q).\notag
   \end{multline}
   Now we set $q=p^\alpha$, $p$ is a prime with $p \ge 3$, $\alpha \in\mathbb{Z}^+$, $x=0$, $(m,p)=1$. Since $\alpha \le k-1$, we have 
   \begin{multline}
   \frac{1}{k}\left(m^kB_k\left(\frac{a}{m}\right)-B_k\right)\\
   \equiv \sum_{\substack{j=0\\(p,a+jm)=1}}^{p^\alpha -1}\left(\left[\frac{a+jm}{p^\alpha }\right]+\frac{1-m}{2}\right)\left(a+jm\right)^{k-1}\quad(\bmod p^\alpha ).\notag
   \end{multline}
   Changing $k$ to $l$, similarly, we get the following congruence
   \begin{multline}
   		\frac{1}{l}\left(m^lB_l\left(\frac{a}{m}\right)-B_l\right)\\
   	\equiv \sum_{\substack{j=0\\(p,a+jm)=1}}^{p^\alpha -1}\left(\left[\frac{a+jm}{p^\alpha }\right]+\frac{1-m}{2}\right)\left(a+jm\right)^{l-1}\quad(\bmod p^\alpha ).\notag
   \end{multline}
   Because $k\nmid p-1$, $l\nmid p-1$ and $k\equiv l\left(\bmod \phi(p^\alpha )\right)$, we can easily get Lemma 2.5 from Lemma 2.1 and Euler's theorem.
    \end{proof}
   
 \begin{lemma}
  Let $n>1$ be a positive integer with $(n,6)=1$, $k\in \mathbb{N}$ and $1 \le k \le \phi_(n^3)-4$, then
  \begin{equation}
  	\sum_{i=1}^{n-1}\frac{\chi_n(i)}{i^k}\equiv \begin{cases} nB_{\phi (n^3)-k,\chi_n}\quad (\bmod n^3)&\text{if $2\mid k,n\in {I(k,2)},$}  \\
  		-\frac{1}{2}kn^2B_{\phi (n^3)-k,\chi_n}\quad (\bmod n^3)&\text{ if $2\nmid k,$}\notag 
  	\end{cases}
  \end{equation}
  where $I(k,\gamma)=\{n > 1:p-1 \nmid k+\gamma, \text{ if $p\mid n$ }\}$.
 \end{lemma}

   \begin{proof}
   By Euler's theorem we have$$ \sum_{i=1}^{n-1}\frac{\chi_n(i)}{i^k}\equiv \sum_{i=1}^{n-1}\chi_n(i)i^{\phi (n^3)-k}\quad (\bmod n^3).$$
   Since $\sum_{i=0}^{N-1} \chi_n(i) i^m=\dfrac{1}{m+1}\left(B_{m+1, \chi_n}(N)-B_{m+1, \chi_n}\right)$ (see \cite[(8)]{Cai}) and the von Staudt-Clausen theorem, we see that
   \begin{align}
   	\sum_{i=1}^{n-1}\chi_n(i)i^{\phi (n^3)-k}={}&\frac{B_{\phi (n^3)-k+1,\chi_n}(n)-B_{\phi (n^3)-k+1,\chi_n}}{\phi (n^3)-k+1}\notag\\
   	={}&\frac{\sum_{i=1}^{\varphi\left(p^3\right)-k+1} \left(\begin{array}{c}
   			\varphi\left(n^3\right)-k+1 \\
   			i
   		\end{array}\right) B_{\varphi\left(n^3\right)-k+1-i,\chi_n}n^i}{\phi (n^3)-k+1}\notag\\
   	={}&\sum_{i=1}^{\varphi\left(n^3\right)-k+1} \frac{1}{i}\left(\begin{array}{c}
   		\varphi\left(n^3\right)-k \\
   		i-1
   	\end{array}\right) B_{\varphi\left(n^3\right)-k+1-i,\chi_n}n^i\notag\\
   	\equiv{}& nB_{\phi (n^3)-k,\chi_n}-\frac{1}{2}kn^2B_{\phi (n^3)-k-1,\chi_n}\quad (\bmod n^3).\notag
   \end{align}
   Noting that $B_{2n+1}=0\left(n\ge 1\right)$, we can complete the  proof of Lemma 2.6.
   \end{proof}
   
   \section{Theorems}
   In order to express the following theorems briefly, we denote$$
   A_m(n,k)=J_m(n)\prod\limits_{p\mid n}\left(1-J_m(p)p^k\right)\frac{B_{k+1}(\frac{1}{m})}{k+1},$$
    $$A(n,k)=\prod\limits_{p\mid n}\left(1-p^k\right)\frac{B_{k+1}(\frac{1}{m})}{k+1},$$
   where $J_m(n)$ is the Jacobi symbol for any integers $n$, $m$. If $(n,m)=1$, then 
   $$J_m(n)=\left(\frac{n}{m}\right)=\begin{cases}1 & \text{if $n\equiv1\quad (\bmod m)$,} \\
   	-1 & \text{if $n\equiv -1\quad (\bmod m).$} 		
   \end{cases}$$
   \begin{thm}
   Let $n$, $m$, $k\in \mathbb{Z}^+$, $(n,6)=1$, $1 \le k < \phi(n^3)-3$ and $m \ge 2$, $n \equiv \pm 1 \left(\bmod m\right)$, $s=\phi(n^3)-k$, we have 
   \begin{multline}
   	\begin{aligned}
   		T_{m,k}(n)\equiv  
   		\begin{cases}
   			-A_m(n,s)+\frac{n}{m}s
   			A(n,s-1)\\ -\frac{k(k+1)n^2}{2m^2}A_m(n,s-2)\quad(\bmod n^3)& \text {if $2\mid k$, $n\in {I'(k,2)},$} \\
   			A(n,s)-\frac{B_{s+1,\chi_n}}{s+1}+\frac{kn}{m}A_m(n,s-1)\\
   			\quad +\frac{k(k+1)n^2}{2m^2}A(n,s-2)\quad(\bmod n^3) &\text {if $2\nmid k$, $(n,k+1)=1,$} \notag
   		\end{cases}
   	\end{aligned}
   \end{multline}
   where $I'(k,\gamma)=\{n > 1$: $p\nmid k$ \text{and $p-1 \nmid k+\gamma$}, \text{if $p\mid n$}\}.
   \end{thm} 
 
    \begin{proof}
    Let $n=dp^l$ with prime $p\ge 5$, $p\nmid d$, $l \ge 1$. Since $n \equiv \pm 1 \left(\bmod m\right)$, the least positive residue of $n$ modulo $m$ is $1$ or $m-1$. By Lemma 2.2(ii) and Lemma 2.4, we get
    \begin{multline}
    \sum_{\substack{x=1 \\p\nmid x}}^{[n/m]} \frac{1}{x^k}\equiv
    \begin{aligned}
    	\begin{cases} -J_m(n)\frac{B_{s'+1}\left(\frac{1}{m}\right)}{s'+1}+\frac{n}{m}B_{s'}\left(\frac{1}{m}\right)\\
    		\quad +J_m(n)\frac{k}{2}\frac{n^2}{m^2}B_{s'-1}\left(\frac{1}{m}\right)\quad (\bmod p^{3l}) & \text {if $2\mid k$, $p^l \in {I(k,2)}$,} \\
    		\frac{B_{s'+1}\left(\frac{1}{m}\right)-B_{s'+1}}{s'+1}-J_m(n)B_{s'}\left(\frac{1}{m}\right)\frac{n}{m}\\
    		\quad -\frac{k}{2}\frac{n^2}{m^2}B_{s'-1}\left(\frac{1}{m}\right)\quad (\bmod p^{3l})& \text {if $2\nmid k$.} 
    	\end{cases}
    \end{aligned}\notag 
    \end{multline}
    Assume $q_1$, $q_2$, \dots, $q_g$ are different prime factors of $d$. Let $s=\phi(n^3)-k$. When $2\mid k$, $n\in {I'(k,2)}$, then
    \begin{align}\tag{3.1}
    	&T_{m,k}(n)=\sum_{\substack{x=1\\p\nmid x}}^{[n/m]}\frac{1}{x^k}-\sum_{i}\sum_{\substack{x=1 \\p\nmid x\\q_i\mid x}}^{[n/m]}\frac{1}{x^k}+\sum_{i,j}\sum_{\substack{x=1 \\p\nmid x\\q_iq_j\mid x}}^{[n/m]} \frac{1}{x^k}+\dots+(-1)^g\sum_{\substack{x=1 \\p\nmid x\\q_1\dots q_g\mid x}}^{[n/m]}\frac{1}{x^k}\notag\\
    	&{}\equiv-J_m(n)\frac{B_{s'+1}\left(\frac{1}{m}\right)}{s'+1}+\frac{n}{m}B_{s'}\left(\frac{1}{m}\right)+J_m(n)\frac{k}{2}\frac{n^2}{m^2}B_{s'-1}\left(\frac{1}{m}\right)\notag\\
    	&{}\quad -\sum_{i}\frac{1}{q_i^k}\left\{-J_m(\frac{n}{q_i})\frac{B_{s'+1}\left(\frac{1}{m}\right)}{s'+1}+\frac{n}{q_im}B_{s'}\left(\frac{1}{m}\right)+J_m(\frac{n}{q_i})\frac{k}{2}\frac{n^2}{q_i^2m^2} B_{s'-1}\left(\frac{1}{m}\right)\right\}\notag\\
    	&{}\quad +\dots +(-1)^g\frac{1}{(q_1\dots q_g)^k}\left\{-J_m(\frac{n}{q_1\dots q_g})\frac{B_{s'+1}\left(\frac{1}{m}\right)}{s'+1}\right.\notag\\
    	&\left.+\frac{n}{q_1\dots q_gm}B_{s'}\left(\frac{1}{m}\right)+J_m(\frac{n}{q_1\dots q_g})\frac{k}{2}\frac{n^2}{(q_1\dots q_g)^2m^2}B_{s'-1}\left(\frac{1}{m}\right)\right\}\notag\\
    	\equiv{}&-J_m(n)\prod_{q \mid \frac{n}{p^l}}\left(1-\frac{1}{J_m(q)q^{k}}\right)\frac{B_{s'+1}\left(\frac{1}{m}\right)}{s'+1}+\frac{n}{m}\prod_{q \mid \frac{n}{p^l}}\left(1-\frac{1}{q^{k+1}}\right)B_{s'}\left(\frac{1}{m}\right)\notag\\
    	&\quad +J_m(n)\frac{k}{2}\frac{n^2}{m^2}\prod_{q \mid \frac{n}{p^l}}\left(1-\frac{1}{J_m(q)q^{k+2}}\right)B_{s'-1}\left(\frac{1}{m}\right)\notag\\
    	\equiv{}& -J_m(n)\prod_{p \mid n}\left(1-J_m(p)p^s\right)\frac{B_{s'+1}\left(\frac{1}{m}\right)}{s'+1}+\frac{n}{m}\prod_{p\mid n}\left(1-p^{s-1}\right)B_{s'}\left(\frac{1}{m}\right)\notag\\
    	&\quad +J_m(n)\frac{k}{2}\frac{n^2}{m^2}\prod_{p \mid n} \left(1-J_m(p)p^{s-2}\right)B_{s'-1}\left(\frac{1}{m}\right)\quad(\bmod n^3).\notag
    \end{align}
    From Lemma 2.5 and Lemma 2.1, we observe 
    \begin{equation}\tag{3.2}
    	\frac{B_{s'+1}\left(\frac{1}{m}\right)}{s'+1}\equiv \frac{B_{s+1}\left(\frac{1}{m}\right)}{s+1}\quad (\bmod p^{3l}),
    \end{equation}
    \begin{equation}\tag{3.3}
    	B_{s'}\left(\frac{1}{m}\right)\equiv\frac{s'}{s} B_{s}\left(\frac{1}{m}\right)\equiv B_{s}\left(\frac{1}{m}\right)\quad (\bmod p^{2l})	
    \end{equation}
    and
    \begin{align}\tag{3.4}
    	B_{s'-1}\left(\frac{1}{m}\right)
    	\equiv{}&\frac{s'-1}{s-1} B_{s-1}\left(\frac{1}{m}\right)\notag \\
    	\equiv{}& \frac{-k-1}{s-1} B_{s-1}\left(\frac{1}{m}\right)\quad (\bmod p^{l}).\notag 	
    \end{align}
    Applying (3.2), (3.3), (3.4) to (3.1), we can get $$
    \begin{aligned}
    	T_{m,k}(n)\equiv& -J_m(n)\prod_{p \mid n}\left(1-J_m(p)p^s\right)\frac{B_{s+1}\left(\frac{1}{m}\right)}{s+1}+\frac{n}{m}\prod_{p\mid n}\left(1-p^{s-1}\right)B_{s}\left(\frac{1}{m}\right)\\
    	&\quad -J_m(n)\frac{k(k+1)}{2}\frac{n^2}{m^2}\prod_{p \mid n} \left(1-J_m(p)p^{s-2}\right)\frac{B_{s-1}\left(\frac{1}{m}\right)}{s-1}\\
    	\equiv& -A_m(n,s)+\frac{n}{m}sA(n,s-1)
    	-\frac{k(k+1)n^2}{2m^2}A_m(n,s-2)\quad (\bmod n^3). 
    \end{aligned}$$\par
    When $2\nmid k$, $(n,k+1)=1$, by Lemma 2.5 and Lemma 2.1 again, we see
    $$
    \frac{B_{s'+1}\left(\frac{1}{m}\right)}{s'+1}\equiv \frac{B_{s+1}\left(\frac{1}{m}\right)}{s+1}\quad (\bmod p^{3l}),
    $$
    $$
    B_{s'}\left(\frac{1}{m}\right)\equiv\frac{s'}{s} B_{s}\left(\frac{1}{m}\right)\equiv \frac{-k}{s} B_{s}\left(\frac{1}{m}\right)\quad (\bmod p^{2l})$$
    and $$
    B_{s'-1}\left(\frac{1}{m}\right)\equiv\frac{s'-1}{s-1} B_{s-1}\left(\frac{1}{m}\right)\equiv  B_{s-1}\left(\frac{1}{m}\right)\quad (\bmod p^{l}).$$
    Using the similar method of the first case, we can obtain the following congruence and complete the proof of Theorem 3.1.
     \begin{align*}
    	&T_{m,k}(n)\equiv \prod_{p\mid n}\left(1-p^s\right)\frac{B_{s+1}\left(\frac{1}{m}\right)-B_{s+1}}{s+1}-\frac{k}{2}\frac{n^2}{m^2}\prod_{p \mid n} \left(1-p^{s-2}\right)B_{s-1}\left(\frac{1}{m}\right)\\
    	&\quad
    	+J_m(n)\frac{kn}{m}\prod_{p\mid n}\left(1-J_m(p)p^{s-1}\right)\frac{B_s\left(\frac{1}{m}\right)}{s} \\
        &{}\equiv A(n,s)-\frac{B_{s+1,\chi_n}}{s+1}+\frac{kn}{m}A_m(n,s-1)
    		+\frac{k(k+1)n^2}{2m^2}A(n,s-2)\quad (\bmod n^3).
    \end{align*}
    \end{proof}
     Putting $m=2, 3, 4, 6$(see \cite[Corollary3.1, Corollary 3.2, Corollary 3.3, Corollary 3.4]{KKU}) win Theorem 3.1, we get the following corollaries.
     
      \begin{cor}
      Let $n>1$ be an odd positive integer, we have
      	\begin{enumerate}[label=\upshape(\roman*), leftmargin=*, widest=ii]
      	\item   $T_{2,1}(n)\equiv-2q_n(2)+n{q_n(2)}^2-\dfrac{2n^2}{3}{q_n(2)}^3-\dfrac{7n^2}{8}B_{\phi(n^3)-2,\chi_n}\quad (\bmod n^3)$, \label{it:5}
      	\item $T_{2,2}(n)\equiv\dfrac{7n}{2}B_{\phi(n^3)-2,\chi_n}+\dfrac{31}{8}n^3B_{\phi(n^3)-4,\chi_n} \quad (\bmod n^3)$. \label{it:6}
      \end{enumerate} 
      \end{cor}
  
     \begin{proof}
     	Note that $B_{2n+1}(\frac{1}{2})=0(n >1)
     	$ and $B_{2n}(\frac{1}{2})=(2^{1-2n}-1)B_{2n}$. From the proofs of Lemma 2.4 and Theorem 3.1, we deduce that$$
     	T_{2,1}(n)\equiv (2^{1-\phi(n^3)}-2)\frac{B_{\phi(n^3),\chi_n}}{\phi(n^3)}-\frac{n^2(2^{3-\phi(n^3)}-1)}{8}B_{\phi(n^3)-2,\chi_n}(\bmod n^3)$$ for $m=2$ and $(n,2)=1$.
     	Now by $2^{\phi(n)}=nq_2(n)+1$, the congruence  
     	\begin{equation}\tag{3.5}
     		\frac{nB_{\phi(n^3)},\chi_n}{\phi(n)}\equiv 1 \quad(\bmod n^3)
     	\end{equation}
     	(which follows from $ pB_{\phi(p^{3l})}\equiv p-1 (\bmod p^{3l})$), and the von Staudt-Clausen theorem, Corollary 3.1(i) is proved.\par
     	Similarly, we have 
     	$$
     	T_{2,2}(n)\equiv\frac{n(2^{3-\phi(n^3)}-1)}{2}B_{\phi(n^3)-2,\chi_n}+\frac{(2^{5-\phi(n^3)}-1)n^3}{8}B_{\phi(n^3)-4,\chi_n}(\bmod n^3)
     	$$for $m=2$ and $(n,2)=1$.
     	Then (ii) follows immediately from the von Staudt-Clausen theorem.
     \end{proof}
       
       \begin{cor}
       	Let $n >1$ be a positive integer with $(n,6)=1$, 
       	we have
       		\begin{enumerate}[label=\upshape(\roman*), leftmargin=*, widest=ii]
       		\item 
       		\begin{multline}
       			T_{3,1}(n)\equiv{} \dfrac{3}{2}\left(-q_3(n)+\dfrac{n}{2}{q_3(n)}^2-\dfrac{n^2}{3}{q_3(n)}^3\right)\\
       			+\dfrac{n}{3}A_3(n,\phi(n^3)-2)-\dfrac{13n^2}{18}B_{\phi(n^3)-2,\chi_n}\quad (\bmod n^3),\notag
       		\end{multline}
       		
       		\item 
       		\begin{multline}
       		T_{3,2}(n)\equiv{}-A_3(n,\phi(n^3)-2)+\dfrac{13n}{3}B_{\phi(n^3)-2,\chi_n}\\
       		-\dfrac{n^2}{3}A_3(n,\phi(n^3)-4) \quad (\bmod n^3)\quad  \text{for $5\nmid n$.}\notag	
       		\end{multline}
       	\end{enumerate}
       \end{cor}
   \begin{proof}
    	By Lemma 2.3, $3^{\phi(n)}=1+nq_3(n)$ and (3.5), we can get $$
    \begin{aligned}
    	\left(\frac{3-3^{\phi(n^3)}}{2\cdot 3^{\phi(n^3)}}-1\right)\frac{B_{\phi(n^3),\chi_n}}{\phi(n^3)}\equiv{}&\frac{3}{2}\cdot \frac{1-3^{\phi(n^3)}}{n^3}\\
    	\equiv{}&\frac{3}{2}\left(-q_3(n)+\frac{n}{2}{q_3(n)}^2-\frac{n^2}{3}{q_3(n)}^3\right)(\bmod n^3).
    \end{aligned} $$
    Taking $m=3$ and $k=1$ in Theorem 3.1, then  using the above congruence, Lemma 2.3 and the von Staudt-Clausen theorem, (i) is obtained.\par
    Taking $m=3$ and $k=2$ in Theorem 3.1, by Lemma 2.3 and the the von Staudt-Clausen theorem again, we complete the proof of (ii).
   \end{proof}
       
      \begin{cor}
      Let $n>1$ be a positive integer with $(n,6)=1$, then
      	\begin{enumerate}[label=\upshape(\roman*), leftmargin=*, widest=ii]
      	\item 
      	\begin{multline}
      	 T_{4,1}(n)\equiv{}3\left(-q_2(n)+\dfrac{n}{2}{q_2(n)}^2-\dfrac{n^2}{3}{q_2(n)}^3\right)\\
      	 -nE(n,\phi(n^3)-2)-\dfrac{7n^2}{8}B_{\phi(n^3)-2,\chi_n}\quad (\bmod n^3),\notag
      	\end{multline}
      	\item
      	\begin{multline} T_{4,2}(n)\equiv{}4E(n,\phi(n^3)-2)+7nB_{\phi(n^3)-2,\chi_n}\\
      		+12n^2E(n,\phi(n^3)-4)\quad (\bmod n^3)\quad \text{for $5\nmid n$,}\notag 
      	\end{multline}
      	
      \end{enumerate} 
      where $E(n,k)=J_4(n)\prod_{p \mid n}\left(1-J_4(p)p^k\right)E_k$, $E_k$ is the $k$th Euler number. 
      \end{cor}
      \begin{proof}
       By Lemma 2.3, $2^{\phi(n)}=1+nq_2(n)$ and (3.5), we have $$
      \begin{aligned}
      	\left(\frac{2-2^{\phi(n^3)}}{ 4^{\phi(n^3)}}-1\right)\frac{B_{\phi(n^3),\chi_n}}{\phi(n^3)} &\equiv\frac{(2^{\phi(n^3)}+2)(1-2^{\phi(n^3)})}{4^{\phi(n^3)}}\frac{B_{\phi(n^3),\chi_n}}{\phi(n^3)}\\
      	&\equiv3\left(-q_2(n)+\frac{n}{2}{q_2(n)}^2-\frac{n^2}{3}{q_2(n)}^3\right)\quad (\bmod n^3).
      \end{aligned}$$
      Taking $m=4$ and $k=1$ in Theorem 3.1, then  using the above congruence, $E_{2 n}=-4^{2 n+1} \dfrac{B_{2 n+1}\left(\frac{1}{4}\right)}{2 n+1}$, Lemma 2.3 and  the von Staudt-Clausen theorem, (i) follows easily.\par
      Taking $m=4$ and $k=2$ in Theorem 3.1, by using  $E_{2 n}=-4^{2 n+1} \dfrac{B_{2 n+1}\left(\frac{1}{4}\right)}{2 n+1}$, Lemma 2.3 and  the von Staudt-Clausen theorem, (ii) is proved.
      \end{proof}
      
      \begin{cor}
      Let $n>1$ be a positive integer with $(n,6)=1$, then 
      	\begin{enumerate}[label=\upshape(\roman*), leftmargin=*, widest=ii]
      	\item  
      	\begin{multline}
      	T_{6,1}(n)\equiv-2q_2(n)+n{q_2(n)}^2-\frac{2n^2}{3}{q_2(n)}^3
      		-\frac{3}{2}q_3(n)+\frac{3n}{4}{q_3(n)}^2-\frac{n^2}{2}{q_3(n)}^3\\
      		+\frac{n}{6}A_6(n,\phi(n^3)-2)-\frac{91n^2}{72}B_{\phi(n^3)-2,\chi_n}\quad (\bmod n^3),\notag
      	\end{multline}
      	\item 
      	\begin{multline}
      	T_{6,2}(n)\equiv -A_6(n,\phi(n^3)-2)+\dfrac{91n}{6}B_{\phi(n^3)-2,\chi_n}\\
      	-\dfrac{n^2}{12}A_6(n,\phi(n^3)-6)\quad (\bmod n^3) \quad \text{ for $5\nmid n$.} \notag
      	\end{multline}
      \end{enumerate} 
      \end{cor}
       \begin{proof}
        From Lemma 2.3,  $2^{\phi(n)}=1+nq_2(n)$,  $3^{\phi(n)}=1+nq_3(n)$ and (3.5), we have $$ 
       \begin{aligned}
       	&\left(\frac{(2-2^{\phi(n^3)})(3-3^{\phi(n^3)})}{2 \cdot 6^{\phi(n^3)}}-1\right)\frac{B_{\phi(n^3),\chi_n}}{\phi(n^3)}\\
       	&\equiv2\cdot \frac{2^{1-\phi(n^3)}}{n^3}+\frac{3}{2}\cdot \frac{1-3^{\phi(n^3)}}{n^3}\\
       	&\equiv-2q_2(n)+n{q_2(n)}^2-\frac{2n^2}{3}{q_2(n)}^3
       	-\frac{3}{2}q_3(n)+\frac{3n}{4}{q_3(n)}^2-\frac{n^2}{2}{q_3(n)}^3\left(\bmod n^3\right).	
       \end{aligned}$$\par
       Taking $m=6$, $k=1$ and $m=6$, $k=2$ in Theorem 3.1, and using the above congruence, we get Corollary 3.4(i) and (ii) respectively in a way similar to Corollary 3.3.
       \end{proof}
    
     \begin{thm}
     	Let $n$, $k$, $m\in \mathbb{Z}^+$, $2\nmid n$,  $1 \le k < \phi(n^3)-3$ and $m \ge 2$, $n \equiv \pm 1 \left(\bmod m\right)$, $s=\phi(n^2)-k$. Then$$
     	T_{m,k}(n)\equiv  \begin{cases}
     		-A_m(n,s)+\frac{n}{m}sA(n,s-1)
     		\quad (\bmod n^2) & \text {if $2\mid k$, $(n,k)=1$,}\\
     		A(n,s)-\frac{B_{s+1,\chi_n}}{s+1}\\
     		\quad+\frac{kn}{m}A_m(n,s-1)\quad (\bmod n^2) &\text {if $2\nmid k,n\in {I(k,1)}$,} 
     	\end{cases}$$
     	where $I(k,\gamma)=\{n > 1$: $p-1 \nmid k+\gamma$, \text{if $p\mid n$}\}.
     \end{thm}
      \begin{proof}
       Let $n=dp^l$ with prime $p\ge 3$, $p\nmid d$, $l \ge 1$, By Lemma 2.4, we have $$
      \sum_{\substack{x=1 \\p\nmid x}}^{[dp^l/m]} \frac{1}{x^k}\equiv \begin{cases} \frac{B_{\phi(p^{2l})-k+1}\left(\left\{\frac{-dp^l}{m}\right\}\right)}{\phi(p^{2l})-k+1}\\
      	\quad +B_{\phi(p^{2l})-k}\left(\left\{\frac{-dp^l}{m}\right\}\right)\frac{dp^l}{m}
      	\quad  (\bmod p^{2l})& \text {if $2\mid k$,}  \\
      	\frac{B_{\phi(p^{2l})-k+1}\left(\left\{\frac{-dp^l}{m}\right\}\right)-B_{\phi(p^{2l})-k+1}}{\phi(p^{2l})-k+1}\\
      	\quad +B_{\phi(p^{2l})-k}\left(\left\{\frac{-dp^l}{m}\right\}\right)\frac{dp^l}{m}
      	\quad(\bmod p^{2l})& \text {if $2\nmid k$, $p^l\in {I(k,1)}$,} 
      \end{cases}
      $$
      From Lemma 2.5 and Lemma 2.1, we see that when $2\mid k$, $(n,k)=1$, $$\frac{B_{\phi(p^{2l})-k+1}\left(\frac{1}{m}\right)}{\phi(p^{2l})-k+1}\equiv \frac{B_{\phi(n^2)-k+1}\left(\frac{1}{m}\right)}{\phi(n^2)-k+1}\quad (\bmod p^{2l})
      $$ and $$
      B_{\phi(p^{2l})-k}\left(\frac{1}{m}\right)\equiv\frac{\phi(p^{2l})-k}{\phi(n^2)-k} B_{\phi(n^2)-k}\left(\frac{1}{m}\right)\equiv B_{\phi(n^2)-k}\left(\frac{1}{m}\right)\quad (\bmod p^{l}),$$
      when $2\nmid k$, $n\in {I(k,1)}$, 
      $$
      \frac{B_{\phi(p^{2l})-k+1}\left(\frac{1}{m}\right)}{\phi(p^{2l})-k+1}\equiv \frac{B_{\phi(n^2)-k+1}\left(\frac{1}{m}\right)}{\phi(n^2)-k+1}\quad (\bmod p^{2l})
      $$ and $$
      \begin{aligned}
      B_{\phi(p^{2l})-k}\left(\frac{1}{m}\right)\equiv&\frac{\phi(p^{2l})-k}{\phi(n^2)-k} B_{\phi(n^2)-k}\left(\frac{1}{m}\right)\\
      \equiv& \frac{-k}{\phi(n^2)-k} B_{\phi(n^2)-k}\left(\frac{1}{m}\right)\quad (\bmod p^{l}).
      \end{aligned}
      $$
      Let $s=\phi(n^2)-k$. Theorem 3.2 follows from a similar argument of theorem 3.1.
      \end{proof}
   
     Putting $m=2, 3, 4$ (see \cite[Corollary3.1, Corollary 3.2, Corollary 3.3]{KKU}), $n=p$ in Theorem 3.2, we deduce the following four corollaries.
     \begin{cor}
     	Let $n>1$ be an odd positive integer, we have 
     	\begin{enumerate}[label=\upshape(\roman*), leftmargin=*, widest=ii]
     		\item   $T_{2,1}(n)\equiv-2q_n(2)+n{q_n(2)}^2\quad (\bmod n^2)$, \label{it:14}
     		\item $T_{2,2}(n)\equiv\dfrac{7n}{2}B_{\phi(n^2)-2,\chi_n} \quad (\bmod n^2).$ \label{it:15}
     	\end{enumerate} 
     \end{cor}
     
     \begin{cor}
     	Let $n>1$ be a positive integer with $(n,6)=1$,
     	we have 
     	\begin{enumerate}[label=\upshape(\roman*), leftmargin=*, widest=ii]
     		\item   $T_{3,1}(n)\equiv \dfrac{3}{2}\left(-q_3(n)+\dfrac{n}{2}{q_3(n)}^2\right)+\dfrac{n}{3}A_3(n,\phi(n^2)-2)\quad (\bmod n^2)$, \label{it:16}
     		\item $T_{3,2}(n)\equiv-A_3(n,\phi(n^2)-2)+\dfrac{13n}{3}B_{\phi(n^2)-2,\chi_n}\quad (\bmod n^2).$ \label{it:17}
     	\end{enumerate}  
     \end{cor}
 
     \begin{cor}
     	Let $n>1$ be an odd positive integer, then
     	\begin{enumerate}[label=\upshape(\roman*), leftmargin=*, widest=ii]
     		\item   $T_{4,1}(n)\equiv-3q_2(n)+\dfrac{3n}{2}{q_2(n)}^2-nE(n,\phi(n^2)-2)\quad (\bmod n^2)\quad  \text{for $3\nmid n$,}$ \label{it:18}
     		\item
     		 $T_{4,2}(n)\equiv4E(n,\phi(n^2)-2)+7nB_{\phi(n^2)-2,\chi_n}\quad (\bmod n^2),$ \label{it:19}
     	\end{enumerate} 
     	where $E(n,k)=J_4(n)\prod_{p \mid n}\left(1-J_4(p)p^k\right)E_k$, $E_k$ is the $k$th Euler number. 
     \end{cor}
 
     \begin{cor}
     For prime $p \equiv \pm 1 \left(\bmod m\right)$, $m > 4$, we have$$
     \sum_{\substack{x=1 \\p\nmid x}}^{[p/m]}\frac{1}{x^2}\equiv -J_m(p)\frac{B_{\phi(p^2)-1}\left(\frac{1}{m}\right)}{\phi(p^2)-1}+\frac{p}{2m}B_{p-2}\left(\frac{1}{m}\right)\quad (\bmod p^2).$$	
     \end{cor}
     \begin{proof}
     	This is a particular case of Theorem 3.2 for $n=p$, $k=2$, $m > 4$. Then by $\phi(p^2)-3>2$, Lemma 2.1 and Lemma 2.5, we can complete the proof of Corollary 3.8. 
     \end{proof}
     Using  the above theorems and corollaries for $m=2$, we have the following results.
     \begin{thm}
     	 For positive integer $n>1$ with $(n,6)=1$, $k\in \mathbb{Z}^+$, we have
     	 \begin{multline}
     		\prod_{d \mid n}\left(\begin{array}{c}
     			kd-1 \\
     			(d-1) / 2
     		\end{array}\right)^{\mu(n / d)}\\ \quad\equiv(-1)^{\frac{\phi(n)}{2}}\left\{4^{k \phi(n)}-\frac{k(7-14k+8k^2)}{4}n^3\frac{B_{\phi(n)-2,\chi_n}}{\phi(n)-2}\right\}\quad (\bmod n^4) \notag
     	\end{multline}
     \end{thm}
    \begin{proof}
    	 Let $$
    	A_n=\left(\begin{array}{c}
    		kn-1 \\
    		(n-1) / 2
    	\end{array}\right),
    	$$
    	then
    	$$
    	A_n=\prod_{r=1}^{(n-1) / 2} \frac{kn-r}{r}=\prod_{d \mid n} \prod_{\substack{r=1 \\(r, n)=d}}^{(n-1) / 2} \frac{kn-r}{r}=\prod_{d \mid n} D_{n / d}=\prod_{d \mid n} D_d,
    	$$
    	in which
    	$$
    	D_d=\prod_{\substack{r=1 \\(r, d)=1}}^{(d-1) / 2} \frac{kd-r}{r} .
    	$$
    	It follows by the multiplicative version of 
    	Möbius inversion formula 
    	\begin{equation}\tag{3.6}
    	D_n=\prod_{d \mid n} A_d^{\mu(n / d)}=\prod_{d \mid n}\left(\begin{array}{c}
    		kd-1 \\
    		(d-1) / 2
    	\end{array}\right)^{\mu(n / d)}.
    	\end{equation}
    	On the other hand, we also have
    	\begin{align}\tag{3.7}
    		D_n=&\prod_{\substack{r=1 \\(r, n)=1}}^{(n-1) / 2} \frac{kn-r}{r} =(-1)^{\phi(n) / 2} \prod_{\substack{r=1 \\(r, n)=1}}^{(n-1) / 2}\left(1-\frac{kn}{r}\right)\notag\\
    		\equiv& (-1)^{\phi(n) / 2}\left\{kn \sum_{\substack{r=1 \\
    				(r, n)=1}}^{(n-1) / 2} \frac{1}{r}+k^2n^2\sum_{\substack{1 \leq r_1<r_2 \leq \frac{n-1}{2} \\ r_1 \neq r_2\\(r_k,n)=1}} \frac{1}{r_1 r_2}\right.\notag\\ &\left.\quad-k^3n^3\sum_{\substack{1 \leq r_1<r_2<r_3 \leq \frac{n-1}{2} \\r_{k_1} \neq r_{k_2}\\(r_k,n)=1}} \frac{1}{r_1 r_2 r_3}\right\}\quad (\bmod n^4).\notag
    	\end{align}
    	In view of 
    	\begin{align}
    		\sum\limits_{\substack{1 \leq r_1<r_2 \leq \frac{n-1}{2} \\ r_1 \neq r_2\\(r_k,n)=1}} \frac{1}{r_1 r_2}=&\frac{1}{2}\sum\limits_{\substack{1 \leq r_1, r_2 \leq \frac{n-1}{2} \\ r_1 \neq r_2}} \frac{1}{r_1 r_2}\notag\\
    		=&\frac{1}{2}\sum_{r_1=1}^{\frac{n-1}{2}} \frac{1}{r_1}\left(\sum_{r_2=1}^{\frac{n-1}{2}} \frac{1}{r_2}-\frac{1}{r_1}\right)
    		=\frac{1}{2}\left(\left(\sum_{r=1}^{\frac{n-1}{2}} \frac{1}{r}\right)^2-\sum_{r=1}^{\frac{n-1}{2}} \frac{1}{r^2}\right)\notag 
    	\end{align}
    	and 
    	\begin{align}
    		\sum_{\substack{1 \leq r_1<r_2<r_3 \leq \frac{n-1}{2} \\r_{k_1} \neq r_{k_2}\\(r_k,n)=1}} \frac{1}{r_1 r_2 r_3}=&\frac{1}{6}\sum_{\substack{1 \leq r_1, r_2, r_3 \leq \frac{n-1}{2} \\r_{k_1} \neq r_{k_2}}} \frac{1}{r_1 r_2 r_3}\notag\\
    		=&\frac{1}{6}\sum_{\substack{1 \leq r_1, i_2 \leq \frac{n-1}{2} \\ r_{k_1} \neq r_{k_2}}} \frac{1}{r_1 r_2}\left(\sum_{\substack{r_3=1}}^{\frac{n-1}{2}} \frac{1}{r_3}-\frac{1}{r_1}-\frac{1}{r_2}\right)\notag\\
    		=&\frac{1}{6}\left(\sum_{\substack{1 \leq r_1,r_2 \leq \frac{n-1}{2} \\ i_1 \neq r_2}} \frac{1}{r_1 r_2} \sum_{r_3=1}^{\frac{n-1}{2}} \frac{1}{r_3}-2 \sum_{\substack{1 \leq r_1, r_2 \leq \frac{n-1}{2} \\ r_1 \neq r_2}} \frac{1}{r_1 r_2^2}\right)\notag\\
    		=&\frac{1}{6}\left(\left(\sum_{r_1=1}^
    		{\frac{n-1}{2}} \frac{1}{r_1}\right)^3-3\sum_{r_1=1}^{\frac{n-1}{2}} \frac{1}{r_1} \sum_{r_2=1}^{\frac{n-1}{2}} \frac{1}{r_2^2}+2 \sum_{r_1=1}^{\frac{n-1}{2}} \frac{1}{r_1^3}\right),\notag 
    	\end{align}
    	(3.7) becomes 
    	\begin{equation}\tag{3.8}
    		\begin{gathered}
    		\begin{aligned}
    			D_n&=\prod_{\substack{r=1 \\(r, n)=1}}^{(n-1) / 2} \frac{kn-r}{r} =(-1)^{\phi(n) / 2} \prod_{\substack{r=1 \\(r, n)=1}}^{(n-1) / 2}\left(1-\frac{kn}{r}\right)\notag\\
    			&\equiv (-1)^{\phi(n) / 2}\left\{1-kn \sum_{\substack{r=1 \\
    					(r, n)=1}}^{(n-1) / 2} \frac{1}{r}+\frac{k^2n^2}{2}\left(\left(\sum_{r=1}^
    			{\frac{n-1}{2}} \frac{1}{r}\right)^2-\sum_{r=1}^{\frac{n-1}{2}} \frac{1}{r^2}\right)\right.\notag\\
    			&\quad\quad \left.-\frac{k^3n^3}{6}\left(\left(\sum_{r=1}^
    			{\frac{n-1}{2}} \frac{1}{r}\right)^3-3\sum_{r=1}^{\frac{n-1}{2}} \frac{1}{r} \sum_{r=1}^{\frac{n-1}{2}} \frac{1}{r^2}+2 \sum_{r=1}^{\frac{n-1}{2}} \frac{1}{r^3}\right)\right\}\quad (\bmod n^4).
    		\end{aligned}	
    		\end{gathered}
    	\end{equation}
    
    	Taking $m=2$, $k=3$ in Theorem 3.1, by using  $B_{2n+1}(\frac{1}{2})=0(n\ge 1) 
    	$, $B_{2n}(\frac{1}{2})=(2^{1-2n}-1)B_{2n}$ and the von Staudt-Clausen theorem, we have 
    	\begin{equation}\tag{3.9}
    		T_{2,3}(n)\equiv (2^{3-\phi(n)}-2)\frac{B_{\phi(n)-2,\chi_n}}{\phi(n)-2}\equiv 6\frac{B_{\phi(n)-2,\chi_n}}{\phi(n)-2}\quad (\bmod n)
    	\end{equation}for $(n,6)=1.$
    	From Corollary 3.5(ii) and the von Staudt-Clausen theorem, we get 
    	\begin{equation}\tag{3.10}
    		T_{2,2}(n)\equiv0\quad (\bmod n)
    	\end{equation}for $(n,6)=1.$
    	Applying Corollary 3.1(i), Corollary 3.5, (3.9) and (3.10) to (3.8), then by Kummer's congruence and $2^{\phi(n)}=1+nq_2(n)$, we have$$
    	\begin{aligned}
    		D_n
    		\equiv{}&(-1)^{\phi(n) / 2}\left\{(1+nq_n(2))^{2k}-\frac{k(7-14k+8k^2)}{4}n^3\frac{B_{\phi(n)-2,\chi_n}}{\phi(n)-2}\right\}\\	     
    		\equiv{}&(-1)^{\frac{\phi(n)}{2}}\left\{4^{k \phi(n)}-\frac{k(7-14k+8k^2)}{4}n^3\frac{B_{\phi(n)-2,\chi_n}}{\phi(n)-2}\right\}\quad (\bmod n^4). 
    	\end{aligned}$$
    	This together with (3.6) proves Theorem 3.3.
    \end{proof}
    \begin{xrem}
    	 \rm{Obviously, (1.7) holds.}
    \end{xrem}
      \begin{cor}[{\cite[Corollary 0.3]{JZC}}]
      	Let	$p > 5$ be a prime, then$$
      	\left(\begin{array}{c}
      		kp-1 \\
      		(p-1) / 2
      	\end{array}\right)\equiv 4^{k (p-1)}+\frac{k(7-14k+8k^2)}{12}p^3B_{p-3}\quad (\bmod p^4)
      	$$
      \end{cor} 
    \begin{proof}
    This is a particular case of Theorem 3.3 for $n=p$, prime $p>5$. From Kummer's congruence, we have $$
    \frac{B_{p-1-2,\chi_n}}{p-1-2}\equiv \frac{B_{p-3}}{-3}\prod_{p \mid n}(1-p^{p-4})\equiv-\frac{B_{p-3}}{3}\quad (\bmod p).$$
    This together with Theorem 3.3 gets Corollary 3.9.
    \end{proof}
       
     \begin{cor}
     If $p>5$ is prime, then for any integer $l \ge 1$, we have
     \begin{multline}
      \left(\begin{array}{c} 
     	p^l-1 \\
     	(p^l-1) / 2
     \end{array}\right)\left/\left(\begin{array}{c}
     	p^{l-1}-1 \\
     	(p^{l-1}-1) / 2
     \end{array}\right)\right.\\
     \equiv(-1)^{\frac{\phi(p^l)}{2}}\left\{4^{ \phi(p^l)}-\frac{1}{4}p^{3l}\frac{B_{\phi(p^l)-2}}{\phi(p^l)-2}\right\}\quad (\bmod p^{4l})\notag
     \end{multline} 
     and $$
     (-1)^{\frac{l(p-1)}{2}}\left(\begin{array}{c}
     	p^l-1 \\
     	(p^l-1) / 2
     \end{array}\right)\equiv 4^{p^l-1}+\frac{4^{p^l-p}}{12}p^3B_{p-3}\quad (\bmod p^4).$$	
     \end{cor}
     \begin{proof}
      The first congruence of Corollary 3.10 is a particular case for $n=p^l$.\par
     We let $l=1, 2, \cdots, l$ and multiply all the resulting congruences, the second one is proved.
     \end{proof}

      \begin{thm}
      	Let $n>1$ be a positive integer with $(n,6)=1$, $k\in \mathbb{Z}^+$, $v$ be any positive odd integer, then
      \begin{multline}
      	\prod_{d \mid n}\left(\begin{array}{c}
      		kd+[vd /2] \\
      		(vd-1) / 2
      	\end{array}\right)^{\mu(n / d)}\\
       \equiv4^{-k \phi(n)}+\frac{k(7v^2+14kv+8k^2)}{4}n^3\frac{B_{\phi(n)-2,\chi_n}}{\phi(n)-2}\quad (\bmod n^4).\notag 
      \end{multline}
   \end{thm} 
     \begin{proof}
     Assume$$
     B_n=\left(\begin{array}{c}
     	kn+[vn /2] \\
     	(vn-1) / 2
     \end{array}\right).
     $$
     Then
     $$
     B_n=\prod_{r=1}^{(vn-1) / 2} \frac{kn+r}{r}=\prod_{d \mid n} \prod_{\substack{r=1 \\(r, n)=d}}^{(vn-1) / 2} \frac{kn+r}{r}=\prod_{d \mid n} S_{n / d}=\prod_{d \mid n} S_d,
     $$
     where
     $$
     S_d=\prod_{\substack{r=1 \\(r, d)=1}}^{(vd-1) / 2} \frac{kd+r}{r}.
     $$
     Using the inversion formula for the Möbius function, we have
     \begin{equation}\tag{3.11}
     	S_n=\prod_{d \mid n} A_d^{\mu(n / d)}=\prod_{d \mid n}\left(\begin{array}{c}
     		kd+[vd /2] \\
     		(vd-1) / 2
     	\end{array}\right)^{\mu(n / d)}.
     \end{equation}
     For $(n,6)=1$, from Theorem 3.1 and Theorem 3.2, we have$$
     T_{2,1}(vn)\equiv2\left(-q_n(2)+\dfrac{n}{2}{q_n(2)}^2-\dfrac{n^2}{3}{q_n(2)}^3\right)-\dfrac{7v^2n^2}{8}B_{\phi(n^3)-2,\chi_n}\quad (\bmod n^3),$$
     $$
     T_{2,1}(vn)\equiv-2q_n(2)+n{q_n(2)}^2\quad (\bmod n^2),$$
     $$
     T_{2,2}(vn)\equiv\dfrac{7vn}{2}B_{\phi(n^2)-2,\chi_n} \quad (\bmod n^2),$$
     $$
     T_{2,1}(vn)\equiv T_{2,1}(n)\quad (\bmod n),$$
     $$
     T_{2,2}(vn)\equiv T_{2,2}(n)\quad (\bmod n),$$
     $$
     T_{2,3}(vn)\equiv T_{2,3}(n)\quad (\bmod n).$$
     Using the above six congruences, Kummer's congruences  and (3.11), we can prove Theorem 3.4 in a way similar to Theorem 3.3. 
     \end{proof}
       
     \begin{cor}[{\cite[Theorem 1.4]{CZC}}]
     For any positive odd integer $k$ and positive integer $n>1$ with $(n,6)=1$, it follows
     $$
      	\prod_{d \mid n}\left(\begin{array}{c}
      		(kd-1) /2 \\
      		(d-1) / 2
      	\end{array}\right)^{\mu(n / d)} \equiv2^{-(k-1) \phi(n)}\quad (\bmod n^3).$$ 
      \end{cor}
    \begin{cor}[{\cite[Corollary 0.2]{JZC}}]
    Let $p > 5$ be a prime and $k$  be a positive integer, then $$
    \left(\begin{array}{c}
    	kp+\frac{p-1}{2}  \\
    	(p-1) / 2
    \end{array}\right)\equiv 4^{-k (p-1)}-\frac{k(7+14k+8k^2)}{12}p^3B_{p-3}\quad (\bmod p^4).$$
    \end{cor}
     \begin{thm}
     Let $n >1$ be a positive integer with $(n,6)=1,$ for any integers $u>v>0$, we have
     \begin{equation}
     	\prod_{d \mid n}\left(\begin{array}{l}
     		u d \\
     		v d
     	\end{array}\right)^{\mu(n / d)} \equiv 1-uv(v-u)\frac{B_{\phi(n)-2,\chi_n}}{\phi(n)-2}n^3 \quad (\bmod n^4).\notag
     \end{equation}	
     \end{thm}
    \begin{proof}
    Define$$ 
    G_n=\left(\begin{array}{l}
    	u n \\
    	v n
    \end{array}\right).
    $$
    Then we have
    $$
    G_n=\frac{u}{v}\left(\begin{array}{c}
    	u n-1 \\
    	v n-1
    \end{array}\right)=\frac{u}{v} \prod_{d \mid n} \prod_{\substack{r=1 \\
    		(r, n)=d}}^{v n-1} \frac{u n-r}{r}=\frac{u}{v} \prod_{d \mid n} V_{n / d}=\frac{u}{v} \prod_{d \mid n} V_d,
    $$
    where
    $$
    V_d=\prod_{\substack{r=1 \\(r, d)=1}}^{v d-1} \frac{u d-r}{r}.
    $$
    Using the inversion formula for the Möbius function, we have
    \begin{equation}\tag{3.12}
    	V_n=\prod_{d \mid n} (\frac{u}{v}G_d)^{\mu(n / d)}=(\frac{u}{v})^{\sum_{d/n}\mu(n / d)}\prod_{d \mid n}G_d^{\mu(n / d)}=\prod_{d \mid n} G_d^{\mu(n / d)}=\prod_{d \mid n}\left(\begin{array}{l}
    		u d \\
    		v d
    	\end{array}\right)^{\mu(n / d)}.
    \end{equation}
    By Euler's theorem, Lemma 2.6 and the von Staudt-Clausen theorem, we get
    \begin{align}\tag{3.13}
    	\sum_{\substack{r=1 \\
    			(r, n)=1}}^{vn-1} \frac{1}{r}&=\sum_{i=0}^{v-1}\sum_{\substack{r=1 \\
    			(r, n)=1}}^{n-1} \frac{1}{r+in}\equiv \sum_{i=0}^{v-1}\sum_{r=1}^{n-1}(r+in)^{\phi (n^3)-1}\notag\\
    	&\equiv\sum_{i=0}^{v-1}\left(\sum_{r=1}^{n-1}\frac{\chi_n(r)}{r}-in\sum_{r=1}^{n-1}\frac{\chi_n(r)}{r^2}+(in)^2\sum_{r=1}^{n-1}\frac{\chi_n(r)}{r^3}\right)\notag \\
    	&\equiv	\sum_{i=0}^{v-1}\left(-\frac{n^2}{2}B_{\phi (n^3)-2,\chi_n}-in^2B_{\phi (n^3)-2,\chi_n}\right)\notag \\
    	&\equiv{-\frac{v^2}{2}n^2B_{\phi (n^3)-2,\chi_n}}\quad (\bmod n^3).\notag
    \end{align}
    Similarly, we have
    \begin{align}\tag{3.14}
    	\sum_{\substack{r=1 \\
    			(r, n)=1}}^{vn-1} \frac{1}{r^2}&=\sum_{i=0}^{v-1}\sum_{\substack{r=1 \\
    			(r, n)=1}}^{n-1} \frac{1}{(r+in)^2}\equiv \sum_{i=0}^{v-1}\sum_{r=1}^{n-1}(r+in)^{\phi (n^2)-2}\notag\\
    	&\equiv\sum_{i=0}^{v-1}\left(\sum_{r=1}^{n-1}\frac{\chi_n(r)}{r^2}-2in\sum_{r=1}^{n-1}\frac{\chi_n(r)}{r^3}\right)\notag \\
    	&\equiv{vnB_{\phi (n^3)-2,\chi_n}}\quad (\bmod n^2)\notag
    \end{align}
    and
    \begin{align}\tag{3.15}
    	\sum_{\substack{r=1 \\
    			(r, n)=1}}^{vn-1} \frac{1}{r}\equiv0\quad (\bmod n).
    \end{align}
    Using (3.13), (3.14), (3.15), (3.12) and Kummer's congruences,  we can prove Theorem 3.5 in a way similar to Theorem 3.3.
    \end{proof}
      \begin{cor}[{\cite[Theorem 3]{Cai}}]
      	For any odd integer $n>1$ with $(n,6)=1$ and integers $u>v>0$, we have$$	
      		\prod_{d \mid n}\left(\begin{array}{l}
      			u d \\
      			v d
      		\end{array}\right)^{\mu(n / d)} \equiv 1 \quad (\bmod n^3).$$
      \end{cor}
  
      \begin{cor}
      Let $p \ge 5$ be a prime, for any integer $u>v>0$, we have
      \begin{equation}
      	\left(\begin{array}{l}
      		u p \\
      		v p
      	\end{array}\right)/\left(\begin{array}{l}
      		u  \\
      		v 
      	\end{array}\right)=\left(\begin{array}{l}
      		u p-1 \\
      		v p-1
      	\end{array}\right) \equiv 1+\frac{uv(v-u)}{3}B_{p-3}p^3 \quad (\bmod p^4).\notag 
      \end{equation}	
      \end{cor} 
      In particular, for $v=1$, $u \ge 2$, we get the same conclusion proved by Glaisher \cite{G90} in 1990. That is 
      \begin{equation}
      	\left(\begin{array}{l}
      		u p-1 \\
      		p-1
      	\end{array}\right) \equiv 1-\frac{u(u-1)}{3}B_{p-3}p^3 \quad (\bmod p^4).\notag 
      \end{equation}
     \normalsize
    \baselineskip=17pt
    \bibliographystyle{alpha}
    \bibliography{reference}
\end{document}